\documentclass[12pt]{amsart}
\usepackage{amsmath,times,epsfig,amssymb,amsbsy,amscd,amsfonts,amstext,color,bm}
\usepackage[margin=1in]{geometry}
\usepackage{enumerate}
\usepackage{placeins}
\usepackage{array}
\usepackage{tabulary}
\usepackage{multirow}
\usepackage{graphicx}
\usepackage{hyperref}
\usepackage{hhline}
\usepackage{soul}
\usepackage{multirow}
\usepackage[table, svgnames, dvipsnames]{xcolor}
\usepackage{makecell, cellspace, caption}
\usepackage{subcaption}
\usepackage[color,matrix,arrow]{xypic}
\usepackage{tikz}
        \usetikzlibrary{arrows.meta}
        \usetikzlibrary{arrows}
        \usetikzlibrary{quotes}
        \usetikzlibrary{graphs}
        \usetikzlibrary{positioning}
        \usetikzlibrary{decorations.markings}
\usepackage{tikz-cd}
\hypersetup{
  colorlinks   = true, 
  urlcolor     = blue, 
  linkcolor    = blue, 
  citecolor   = red 
}

\theoremstyle{plain}
\newtheorem{theorem}{Theorem}[section]
\newtheorem{corollary}[theorem]{Corollary}
\newtheorem{lemma}[theorem]{Lemma}

\newtheorem{proposition}[theorem]{Proposition}

\makeatletter
    
    \@addtoreset{equation}{section}
  \makeatother

\theoremstyle{definition}
\newtheorem{definition}[theorem]{Definition}
\newtheorem{example}[theorem]{Example}

\theoremstyle{remark}
\newtheorem{remark}[theorem]{Remark}

\newcommand{\A}{\mathcal{A}}
\newcommand{\B}{\mathcal{B}}

\newcommand{\scF}{\mathcal{F}}

\newcommand{\scC}{\mathcal{C}}
\newcommand{\SC}{\mathcal{SC}}

\newcommand{\R}{\mathbb{R}}

\newcommand{\Z}{\mathbb{Z}}

\newcommand{\V}{{\mathcal{V}}}
\newcommand{\E}{{\mathcal{E}}}

\newcommand{\vn}{\noindent}

\newcommand{\I}{{\mathcal{I}}}

\newcommand{\Br}{\operatorname{Br}}

\newcommand{\supp}{\operatorname{supp}}

\newcolumntype{K}[1]{>{\centering\arraybackslash}p{#1}}

\begin{document}

\title[Worpitzky-compatible arrangements and cocomparability graphs]{Worpitzky-compatible subarrangements of braid arrangements and cocomparability graphs}

\date{\today}

\begin{abstract}
The class of Worpitzky-compatible subarrangements of a Weyl arrangement together with an associated Eulerian polynomial was recently introduced by Ashraf, Yoshinaga and the first author, which brings the characteristic and Ehrhart quasi-polynomials into one formula. The subarrangements of the braid arrangement, the Weyl arrangement of type $A$, are known as the graphic arrangements. We prove that the Worpitzky-compatible graphic arrangements are characterized by cocomparability graphs. Our main result yields new formulas for the chromatic and graphic Eulerian polynomials of cocomparability graphs. 
 \end{abstract}

\author{Tan Nhat Tran}
\address{Tan Nhat Tran, Department of Mathematics, Hokkaido University, Kita 10, Nishi 8, Kita-Ku, Sapporo 060-0810, Japan.}
\email{trannhattan@math.sci.hokudai.ac.jp}

\author{Akiyoshi Tsuchiya}
\address{Akiyoshi Tsuchiya, Graduate School of Mathematical Sciences, University of Tokyo, 3-8-1 Komaba, Meguro-ku, Tokyo 153-8914, Japan.}
\email{akiyoshi@ms.u-tokyo.ac.jp}


\date{\today}
\maketitle

\section{Introduction}
\label{sec:intro}

Let   $V$ be an $\ell$-dimensional Euclidean vector space with the standard inner product $(\cdot,\cdot)$.
 Let $\Phi$ be an irreducible (crystallographic) root system in $V$, with a fixed positive system $\Phi^+ \subseteq \Phi$ and the associated set of simple roots $\Delta := \{\alpha_1,\ldots,\alpha_\ell \}$. 
 For $m \in \Z$ and $\alpha \in  \Phi$, define the \emph{affine hyperplane} $H_{\alpha,m}$ by
$H_{\alpha,m} :=\{x\in V \mid(\alpha,x)=m\}.$ 
 For $\Psi\subseteq\Phi^+$, the \emph{Weyl subarrangement} $\A_{\Psi}$ of $\Psi$ is  defined by $\A_{\Psi}:= \{H_{\alpha,0} \mid \alpha\in\Psi\}$. 
In particular, $\A_{\Phi^+}$ is called the \emph{Weyl arrangement}. 
Define the partial order $\ge$ on $\Phi^+$ as follows: $\beta_1 \ge \beta_2$ if $\beta_1-\beta_2 \in\sum_{i=1}^\ell \Z_{\ge 0}\alpha_i$. 
A subset $\Psi\subseteq\Phi^+$ is an \emph{ideal} of $\Phi^+$  if for $\beta_1,\beta_2 \in \Phi^+$, $\beta_1 \ge \beta_2, \beta_ 1 \in \Psi$ implies $\beta_2 \in \Psi$. 
For an ideal $I\subseteq\Phi^+$, the corresponding Weyl subarrangement $\A_{I}$ is called the \emph{ideal subarrangement}.

We will be mainly interested in the case $\Phi$ of type $A_{\ell-1}$, in which $\A_{\Phi^+}$ is widely known as the \emph{braid arrangement}, denoted $\Br(\ell)$. 
We recall a popular construction of type $A$ root systems.
Let $\{\epsilon_1, \ldots, \epsilon_{\ell}\}$ be an orthonormal basis for $V$, and define $U : = \{ \sum_{i=1}^{\ell} r_i\epsilon_i  \in V\mid \sum_{i=1}^{\ell} r_i=0\} \simeq \R^{\ell-1}$. 
 The set $\Phi(A_{\ell-1}) = \{\pm(\epsilon_i - \epsilon_j) \mid 1 \le i<j \le \ell\}$ 
is a root system of type $A_{\ell-1}$ in $U$, with a positive system
$
\Phi^+(A_{\ell-1}) =  \{ \epsilon_i-\epsilon_j \mid 1 \le i <  j \le \ell\}
$
 and the associated set of simple roots $\Delta(A_{\ell-1}) = \{\alpha_i:=\epsilon_i - \epsilon_{i+1} \mid 1 \le i  \le \ell-1\}$.
Thus, a subarrangement $\B$ of $\Br(\ell)$ is completely defined by a simple graph $G = ([\ell],\E)$, where $\{x_i-x_j=0\} \in \B$ if and only if $\{i,j\}  \in \E$.
 Given a graph $G$, let $\A(G)$ be the arrangement that it defines, or the corresponding \emph{graphic arrangement}.
It is a standard fact that $\A(G)$ is the product (e.g., \cite[Definition 2.13]{OT92}) of the one dimensional empty arrangement and the Weyl subarrangement $\A_{\Psi(G)}$, where $\Psi(G) := \{  \epsilon_i-\epsilon_j  \mid \{i,j\}  \in \E \,( i <  j)\} \subseteq \Phi^+(A_{\ell-1}).$
Throughout the paper, for any property that $\A_{\Psi(G)}$ has, we will say $\A(G)$ has that property as well.

Weyl arrangements are an important class of \emph{free} arrangements in the sense of Terao. 
In other words, the arrangement's logarithmic derivation module is a free module \cite[\S4, \S6]{OT92}. 
There has been considerable interest  in analyzing subarrangements of a Weyl arrangement from the perspective of freeness. 
A central hyperplane arrangement is \emph{supersolvable} if its intersection lattice is supersolvable in the sense of Stanley \cite{St72}. 
Jambu-Terao  proved that any supersolvable arrangement is free \cite{JT84}. 
Various free subarrangements of a Weyl arrangement of  type $B$  were studied, e.g., \cite{JS93, ER94, STT19}. 
Remarkably, a striking result of Abe-Barakat-Cuntz-Hoge-Terao \cite{ABCHT16} asserts that any ideal subarrangement is free. 

Although characterizing free subarrangements of an arbitrary Weyl arrangement is still a challenging problem, in the case of braid arrangement, the free subarrangements can be completely analyzed using the connection to graphs. 
It follows from the works of Stanley \cite{St72} and Edelman-Reiner \cite{ER94} that free and supersolvable graphic arrangements  are synonyms, and they correspond to \emph{chordal graphs} (every induced cycle in the graph has exactly $3$ vertices). 
Chordal graphs are a superclass of \emph{(unit) interval graphs} (each vertex can be associated with an (unit) interval on the real line, and two vertices are adjacent if the associated intervals have a nonempty intersection). 
More strongly, a graph is an interval graph if and only if it is a \emph{cocomparability} (Definition \ref{def:c-co}) and chordal graph \cite{GH64} (see Figure \ref{fig:relation} for an illustration). 
It is also notable that the ideal subarrangements of a braid arrangement are parametrized by unit interval graphs 
(Theorem  \ref{thm:ideal-characterize}). 
 \begin{definition}
 \label{def:c-co}
 A graph is called a  \emph{comparability graph} if its edges can be transitively oriented, i.e., if $u \to v$ and $v \to w$, then $u \to w$. 
  A graph is called a  \emph{cocomparability graph} if its complement is a comparability graph.
\end{definition}


Thus, it is natural to ask which class of Weyl subarrangements generalizes the cocomparability graphs.
Recently, the notion of \emph{(Worpitzky-)compatible arrangements} was introduced by Ashraf, Yoshinaga and the first author in the study of characteristic quasi-polynomials of Weyl subarrangements and Ehrhart theory \cite{ATY20}. It is shown that any ideal subarrangement is compatible \cite[Theorem 4.16]{ATY20}.
In this paper we prove that, interestingly,  the Worpitzky-compatible graphic arrangements are characterized by cocomparability graphs, which gives an answer to the aforementioned question.

To state the result formally, we first recall the concept of compatibility. 
A connected component of $V \setminus \bigcup_{ \alpha\in \Phi^+,m \in \Z} H_{\alpha,m}$ is called an \emph{alcove}.
Let $A$ be an alcove. 
The \emph{walls} of $A$ are the hyperplanes that support a facet of $A$. 
The \emph{ceilings} of $A$ are the walls which do not pass through the origin and have the origin on the same side as $A$. 
The  \emph{upper closure} $A^\diamondsuit$  of $A$ is the union of $A$ and its facets supported by the ceilings of $A$.
Let  $P^\diamondsuit:=\{x\in V \mid 0<(\alpha_i ,x)\le 1\,(1 \le i \le \ell)\}$ be the \emph{fundamental parallelepiped} (of the coweight lattice) of $\Phi$. 
Thus, 
$$P^\diamondsuit=\bigsqcup_{A:\, \text{alcove},\,A \subseteq P^\diamondsuit}A^\diamondsuit,
$$ 
which is known as the \emph{Worpitzky partition}, e.g., \cite[Proposition 2.5]{Y18W}, \cite[Exercise 4.3]{H90}.

 \begin{definition}{(\cite[Definition 4.8]{ATY20})}
 \label{def:compatibleW}   
A subset $\Psi \subseteq \Phi^+$ is said to be \emph{Worpitzky-compatible} (or \emph{compatible} for short) if every nonempty intersection of the upper closure $A^\diamondsuit$ of an alcove $A\subseteq P^\diamondsuit$ and an affine hyperplane w.r.t. a root in $\Psi$ can be lifted to a facet intersection.
That is, $A^\diamondsuit \cap H_{\alpha, m_\alpha}$ for $\alpha\in \Psi,m_\alpha \in \Z$ is either empty, or contained in a ceiling $H_{\beta, m_\beta}$ of $A$  with $\beta\in \Psi,m_\beta \in \Z$. 
If  $\Psi$ is compatible, the Weyl subarrangement $\A_{\Psi}$ is said to be \emph{compatible} as well.

\end{definition}

Our main result is the following.
\begin{theorem}
 \label{thm:cocomp-characterize}
 Let  $G = (\V,\E)$ be a graph with $|\V|=\ell$.  The following  are equivalent. 
\begin{enumerate}[(i)] 
\item $G$ has a labeling using elements from $[\ell]$ so that $\A(G)$ is a compatible graphic arrangement. 
\item  $G$ has an ordering  $v_1 < v_2  < \cdots  <  v_\ell$ of its vertices such that if $i   < j$, $\{v_i, v_j\} \in \E$ and $(p_1, \ldots, p_m)$ is any sequence such that $i = p_1 < p_2 <\cdots < p_m=j$, then there exists $p_a$ with $1 \le a < m$ such that $\{v_{p_a},v_{p_{a+1}}\}  \in \E$. 
\item $G$ is a cocomparability graph.
    \end{enumerate}
\end{theorem}

In particular, our main result yields a new characterization of cocomparability graphs and in turn contributes to the second row and the fourth row of Table \ref{tab:concepts}.

\begin{table}[htbp]
\centering
{\renewcommand\arraystretch{1.5} 
\begin{tabular}{c|c|c}
Graph class & Weyl subarrangement class & Location \\
\hline\hline
cocomparability &  compatible (= strongly compatible) & Theorems  \ref{thm:cocomp-characterize},  \ref{thm:C=SC-A}\\
\hline
chordal  &  free (= supersolvable)   & \cite{St72,ER94} \\
\hline
interval &   compatible $\cap$ free  & Corollary \ref{cor:interval}\\
\hline
unit interval &  ideal & Theorem  \ref{thm:ideal-characterize}\\
\end{tabular}
}
\bigskip
\caption{Parallel concepts in type $A$.}
\label{tab:concepts}
\end{table}

\def\firstcircle{(0,0) circle (2.5cm)}
\def\secondcircle{(2,0) circle (2.5cm)}
\def\thirdcircle{((1,0) ellipse (1.2cm and 1cm)}
\def\fourthcircle{((-1,0)  ellipse (1.5cm and 1cm)}
\def\fifthcircle{((1.5,0) ellipse (2cm and 1.5cm)}

\begin{figure}[htbp]
\centering
\begin{tikzpicture}
    \draw \firstcircle;
    \draw \secondcircle ;
    \draw \thirdcircle node  {U};
    \node at (1,1.8)    {I};    
    \node at (-1.7,1)    {Ch};
    \node at (3.8,1)    {Co};
 
\end{tikzpicture}
\caption{Relationship between graph classes: $\text{U} \subsetneq \text{I} = \text{Ch} \cap \text{Co}$. Co: cocomparability, Ch: chordal, I: interval, U: unit interval.}
\label{fig:relation}
\end{figure}
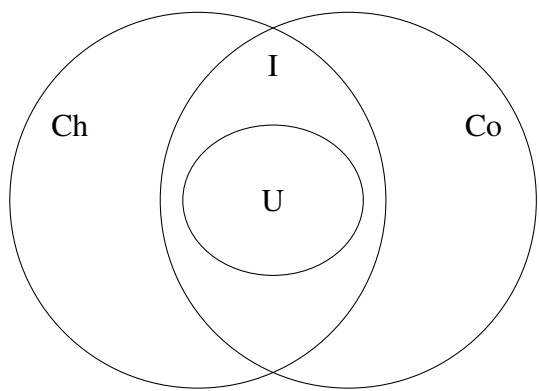

\begin{figure}[htbp]
\centering
\begin{tikzpicture}
    \draw \firstcircle;
    \draw[red] \secondcircle ;
    \draw \thirdcircle node  {$\I$};
        \draw \fourthcircle ;
            \draw[red] \fifthcircle ;
    \node at (-1.8,.5)    {$\mathcal{SS}$};
    \node at (-1, 2)    {$\scF$};
    \node at (3,2)    {\color{red}{$\scC$}};
        \node at (3,.5)    {\color{red}{$\SC$}};

\end{tikzpicture}
\caption{Relationship  between Weyl subarrangement  classes. $\scC$: compatible, $\SC$: strongly compatible, $\scF$: free, $\mathcal{SS}$: supersolvable, $\I$: ideal.}
\label{fig:relation-arr}
\end{figure}
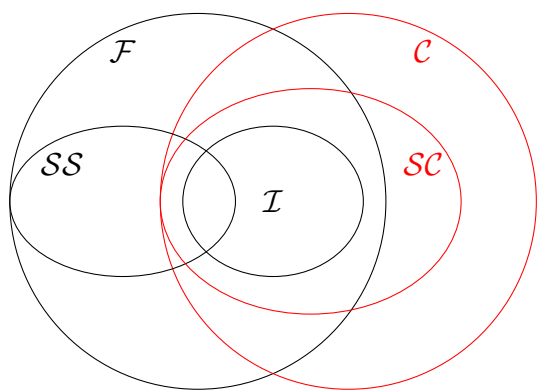

\section{Proof of the main result}
\label{sec:proof}

First, we introduce a new subclass of Worpitzky-compatible sets, which will play a key role in the proof of our main result.
 \begin{definition}
 \label{def:strongly-compatible}   
A subset $\Psi \subseteq \Phi^+$ is said to be \emph{strongly (Worpitzky-)compatible} if for every $\alpha \in \Psi$ and  for any choice of positive roots $\beta_1, \ldots, \beta_m \in \Phi^+$  such that $\alpha \in \sum_{i=1}^m \Z_{>0}\beta_i$, there exists $k$ with $1 \le k \le m$ such that $\beta_k \in \Psi$.
\end{definition}

The definition above was made by inspiration of the following lemma, which is an important result in \cite{ATY20}.
 \begin{lemma}
\label{lem:ideal-crucial}
Let $A  \subseteq P^\diamondsuit$ be an alcove.  
If there exist $\alpha\in \Phi^+,r_\alpha \in \Z$ so that $A^\diamondsuit \cap H_{\alpha,r_\alpha}= \bigcap_{j=1}^m H_{\beta_j, r_{\beta_j}}\cap  A^\diamondsuit$ is a  face   of $A^\diamondsuit$,  where $H_{\beta_j, r_{\beta_j}}$ are the ceilings of $A$, then
$\alpha \in \sum_{j=1}^m \Z_{\ge0}\beta_j$. 
\end{lemma} 
\begin{proof} 
 See  \cite[Proof of Theorem 4.16]{ATY20}. 
\end{proof}

Let $\I$ be the set of all ideals of $\Phi^+$. 
In addition, let $\scC$ (resp,. $\SC$) be the set of all compatible  (resp., strongly compatible) sets of $\Phi^+$. 
We exhibit a relation between these sets (see also Figure \ref{fig:relation-arr} for an illustration).
 \begin{theorem}
\label{thm:ideal-strong-com}
If $\Phi$ is an irreducible root system, then
$$\I \subseteq \SC \subseteq \scC.$$
 \end{theorem}
  \begin{proof} 
The first inclusion is clear. 
  The second inclusion follows from Lemma \ref{lem:ideal-crucial}. 
\end{proof}

 \begin{remark}
\label{rem:counter-ex}
Example \ref{ex:Claw} illustrates a strongly compatible set but not an ideal (w.r.t. any positive system of $\Phi$). 
There exists a  compatible set that is not strongly compatible when $\Phi$ is of type $G_2$ (or $B_2$) \cite[Example 4.18(d)]{ATY20}. 
\end{remark}

For any alcove $A$ and  $\gamma\in \Phi^+$, there exists a unique integer $r$ with $r-1 < (x,\gamma) < r$ for all $x \in A$. We denote this integer by $r(A,\gamma)$.
 \begin{lemma}
\label{lem:shi}
Suppose that for each $\gamma\in \Phi^+$ we are given a positive integer $r_\gamma$. 
There is an alcove $A$ with $r(A,\gamma)=r_\gamma$ for all  $\gamma\in \Phi^+$
if and only if $r_\gamma+r_{\gamma'}-1 \le r_{\gamma+\gamma'} \le r_\gamma+r_{\gamma'}$ whenever $\gamma,\gamma',\gamma+\gamma'\in \Phi^+$.
\end{lemma} 
\begin{proof} 
This was first proved by Shi in terms of coroots \cite[Theorem 5.2]{Shi87}. 
The statement here is formulated in terms of roots, which can be found in, e.g., \cite[Lemma 2.4]{A05}.
\end{proof}

Surprisingly, there is no difference between compatible and strongly compatible sets in the case of type $A$.
 \begin{theorem}
\label{thm:C=SC-A}
If $\Phi$ is of type $A_\ell$, then every compatible set is strongly compatible, i.e.,
$$\SC= \scC.$$
 \end{theorem}
  \begin{proof} 
 
Let  $\Psi \subseteq \Phi^+$ be a  compatible set. 
Let $\alpha \in \Psi$, and suppose that there are $\beta_1, \ldots, \beta_m \in \Phi^+$  such that $\alpha \in \sum_{i=1}^m \Z_{>0}\beta_i$. 
Thus, $\alpha=\sum_{i=1}^m \beta_i$ = $\alpha_s+\alpha_{s+1}+\ldots+\alpha_t$ for $1 \le s \le t \le \ell$. 
We want to show that there exists $k$ with $1 \le k \le m$ such that $\beta_k \in \Psi$.

For $\beta= \sum_{i=1}^\ell d_i \alpha_i\in   \Phi^+$, denote $\supp(\beta):=\{\alpha_i \mid d_i>0\}$. 
The above expression of $\alpha$ induces a partition of $\Delta$ as follows:
$$\Delta = S_0 \sqcup S_1 \sqcup \cdots \sqcup S_m \sqcup S_{m+1},$$
where $S_0 := \{\alpha_1, \ldots, \alpha_{s-1}\}$, $S_{m+1} := \{\alpha_{t+1}, \ldots, \alpha_{\ell}\}$, $S_i := \supp(\beta_i)$ for $1 \le i \le m$.
For each $\gamma\in \Phi^+$, define
$$r_\gamma:=\#\{0 \le i \le m+1 \mid  \supp(\gamma) \cap S_i \ne \emptyset\}.$$
For any $\gamma,\gamma'\in \Phi^+$ with $\gamma+\gamma'\in \Phi^+$, it is not hard to show the following facts:
\begin{enumerate}[(a)]
\item  $r_{\gamma+\gamma'} = r_\gamma+r_{\gamma'}$ if there is no $i$ with $0 \le i \le m+1$ such that $\supp(\gamma) \cap S_i\ne \emptyset$ and $\supp(\gamma') \cap S_i\ne \emptyset$,
\item  $r_{\gamma+\gamma'} = r_\gamma+r_{\gamma'}-1$, otherwise.
\end{enumerate}
By Lemma \ref{lem:shi}, there is an alcove $A$ with $r(A,\gamma)=r_\gamma$ for all  $\gamma\in \Phi^+$. 
In particular, $r(A,\alpha)=m$ and $r(A,\beta_i)=1$ for $1 \le i \le m$.
Clearly, $A \subseteq P^\diamondsuit$.

Fix $ \beta_i$ with $1 \le i \le m$. 
We will show that $H_{\beta_i,1}$ is a ceiling of $A$. 
The method used in \cite[Theorem 3.11]{A05} or \cite[Theorem 3.1]{Th14} applies here as well.
First, the above formulas imply that 
\begin{enumerate}[(a)]
\item  $r_{\beta_i+\gamma} = r_\gamma+r_{\beta_i}$ if $\gamma,\beta_i+\gamma\in \Phi^+$ and
\item  $r_{\beta_i} = r_\gamma+r_{\gamma'}-1$ if $\beta_i = \gamma+\gamma'$ for $\gamma,\gamma'\in \Phi^+$.
\end{enumerate}
Next, we again apply Lemma \ref{lem:shi} to have an alcove $B$ with $r(B,{\beta_i})=r_{\beta_i}+1=2$, $r(B,\gamma)=r_\gamma$ for all  $\gamma\in \Phi^+\setminus \{\beta_i\}$. 
This implies that $H_{\beta_i,1}$ is a wall, hence a ceiling of $A$. 
Thus, the non-empty face $\bigcap_{i=1}^m H_{\beta_i,1}\cap  A^\diamondsuit$ is contained in the face $H_{\alpha,m}\cap  A^\diamondsuit$ of $A^\diamondsuit$. 
Since any proper face of a polytope is the intersection of all facets containing it, we must have $H_{\alpha,m}\cap  A^\diamondsuit=\bigcap_{i=1}^m H_{\beta_i,1}\cap  A^\diamondsuit$. 
Since $\Psi$ is a  compatible set, by Definition \ref{def:compatibleW}, there exists $k$ with $1 \le k \le m$ such that $\beta_k \in \Psi$. 
This completes the proof.
\end{proof}

Thanks to Theorem \ref{thm:C=SC-A}, when $\Phi$ is of type $A$, the concept of compatibility originally defined by alcoves and affine hyperplanes can now be rephrased in terms of graphs. 
 
\begin{corollary}\label{cor:c-A}   
Let $G = ([\ell],\E)$ be a graph.  The following  are equivalent. 
\begin{enumerate}[(i)] 
\item $\A(G)$ is a compatible graphic arrangement, i.e., $\Psi(G)\subseteq \Phi^+(A_{\ell-1})$ is a compatible set. 
\item $G$ has the following property: if $i   < j$, $\{i,j\}   \in \E $ and $(p_1, \ldots, p_m)$ is any sequence such that $i = p_1 < p_2 <\cdots < p_m=j$, then there exists $p_a$ with $1 \le a < m$ such that $(p_a,p_{a+1})  \in \E $.
    \end{enumerate}
\end{corollary} 
\begin{proof} 
The implication $(i) \Rightarrow (ii)$ is clear.
Now we prove $(ii) \Rightarrow (i)$.
Let $\alpha=\epsilon_i - \epsilon_j   \in \Psi(G)$. 
If there are $\beta_1, \ldots, \beta_m\in \Phi^+$ such that $\alpha \in \sum_{i=1}^m \Z_{>0}\beta_i$, then $\alpha = \sum_{\beta \in M} \beta$ where $M=\{\beta_1, \ldots, \beta_m\}$. 
Therefore, $\epsilon_i - \epsilon_j  = \sum_{p \in H}\epsilon_p -  \sum_{q \in T}\epsilon_q$ for  $H, T \subseteq [\ell]$. 
Note that $\epsilon_i - \epsilon_j \ge \epsilon_p - \epsilon_q$ if and only if $i \le p < q \le j$. 
The independence of $\{\epsilon_1, \ldots, \epsilon_{\ell}\}$ implies $H\setminus T= \{i\}$ and $T\setminus H = \{j\}.$
Thus, there exist $p_1, \ldots, p_{m+1} \in \Z_{>0}$ with $i = p_1 < p_2 <\cdots <  p_{m+1} =j$ such that 
$M = \{ \epsilon_{p_{a}}- \epsilon_{p_{a+1}} \mid 1 \le a \le m\}$.
Therefore, there is $p_a$  such that $(p_a,p_{a+1})  \in \E $ hence $ \epsilon_{p_{a}}- \epsilon_{p_{a+1}}  \in \Psi(G) \cap M$.
This completes the proof.
\end{proof}

 \begin{definition}
 \label{def:umbrella-free}   
 An ordering  $v_1 < v_2  < \cdots  <  v_\ell$ of the vertices of a graph $G$ is an \emph{umbrella-free ordering} (or a \emph{cocomparability ordering}) if $ i < k < j$ and $\{v_i, v_j\} \in \E$, then either $\{v_i, v_k\} \in \E$ or $\{v_k, v_j\} \in \E$ or both.
\end{definition}

 \begin{lemma}
 \label{lem:cocomp-characterize1}
 Let  $G = (\V,\E)$ be a graph  with $|\V|=\ell$.  The following  are equivalent. 
\begin{enumerate}[(i)] 
\item $G$ is a cocomparability graph.
\item $G$ has an  umbrella-free ordering.
\item  $G$ has an ordering  $v_1 < v_2  < \cdots  <  v_\ell$ of its vertices such that  if $i   < j$, $\{v_i, v_j\} \in \E$ and $(p_1, \ldots, p_m)$ is any  sequence such that $i = p_1 < p_2 <\cdots < p_m=j$, then there exists $p_a$ with $1 \le a < m$ such that $\{v_{p_a},v_{p_{a+1}}\}  \in \E$. 
In other words,  if we label vertex $v_i$ by $i\in[\ell]$, then $G$ has the property in Corollary \ref{cor:c-A}.
    \end{enumerate}
\end{lemma}
\begin{proof} 
The equivalence $(i) \Leftrightarrow (ii)$ is well known, e.g., \cite[\S2]{KS93}. 
We prove the equivalence $(ii) \Leftrightarrow (iii)$ by showing that ``$<$" is an umbrella-free ordering if and only if ``$<$"  satisfies the condition in (iii). 
The implication ($\Leftarrow$) is obvious. 
To prove ($\Rightarrow$), suppose to the contrary that $\{v_{p_a},v_{p_{a+1}}\}  \notin \E$ for all $1 \le a < m$. 
Since $p_1 < p_{m-1} < p_m$, we must have $\{v_{p_1},v_{p_{m-1}}\}  \in \E$. 
Repeating yields $\{v_{p_1},v_{p_{3}}\}  \in \E$ with  $p_1 < p_{2} < p_3$, $\{v_{p_1},v_{p_{2}}\} \notin \E$, $\{v_{p_2},v_{p_{3}}\}  \notin \E$. 
This is a contradiction. 
 \end{proof}

\begin{remark}
\label{rem:fis}
Comparability and cocomparability graphs also have forbidden induced subgraph characterizations, see, e.g.,  \cite{Ga67, Tr92}.
\end{remark}

We are now in position to prove our main result.
\begin{proof}[\textbf{Proof of Theorem \ref{thm:cocomp-characterize}}]
The equivalence $(i) \Leftrightarrow (ii)$ follows from Corollary \ref{cor:c-A}. 
The equivalence $(ii) \Leftrightarrow (iii)$ follows from Lemma \ref{lem:cocomp-characterize1}.
 \end{proof}
 
 \begin{remark}
\label{rem:unlike-free}
By \cite[Theorem 3.3]{ER94}, a graph $G$ is chordal if and only if  $\A(G)$ is free (= supersolvable) for \emph{any} labeling of $G$ using $[\ell]$ (since if $G$ and $G'$ are isomorphic, then $\A(G)$ and $\A(G')$ have isomorphic intersection lattices). 
Given a cocomparability graph $G$, Theorem \ref{thm:cocomp-characterize}, however, can only tell the existence of a labeling that makes $\A(G)$ compatible  (see Examples \ref{ex:G} and \ref{ex:G'}). 
\end{remark}

As a consequence, we obtain a characterization of  graphic arrangements that are compatible and free.
\begin{corollary}
 \label{cor:interval}
 Let  $G = (\V,\E)$ be a graph with $|\V|=\ell$.  The following  are equivalent. 
\begin{enumerate}[(i)] 
\item $G$ has a labeling using $[\ell]$ so that $\A(G)$ is a compatible and free graphic arrangement. 
\item  $G$ has an ordering  $v_1    < \cdots  < v_\ell$ of its vertices such that  if $ i < k < j$ and $\{v_i, v_j\} \in \E$, then $\{v_i, v_k\} \in \E$.
\item $G$ is an interval graph.
    \end{enumerate}
\end{corollary}

\begin{proof} 
The equivalence $(ii) \Leftrightarrow (iii)$ is well known, e.g., \cite[Theorem 4]{Ol91}. 
The equivalence $(i) \Leftrightarrow  (iii)$ follows from Theorem \ref{thm:cocomp-characterize}, Remark \ref{rem:unlike-free}, and the fact that a graph is an interval graph if and only if it is a cocomparability and chordal graph \cite[Theorem 2]{GH64}.
 \end{proof}
 
We complete Table \ref{tab:concepts} by giving a graphic characterization of ideal subarrangements of braid arrangements, which was suggested to us by Shuhei Tsujie.

 \begin{theorem} 
 \label{thm:ideal-characterize}
 Let  $G = (\V,\E)$ be a graph with $|\V|=\ell$.  The following  are equivalent. 
\begin{enumerate}[(i)]
\item $G$ has a labeling using $[\ell]$ so that $\A(G)$ is an ideal-graphic arrangement, i.e., $\Psi(G)\subseteq \Phi^+(A_{\ell-1})$ is an ideal. 
\item  $G$ has an ordering  $v_1    < \cdots  < v_\ell$ of its vertices such that  if $ i < k < j$ and $\{v_i, v_j\} \in \E$, then $\{v_i, v_k\} \in \E$ and $\{v_k, v_j\} \in \E$.
\item   $G$  is  an unit interval graph.
    \end{enumerate}

\end{theorem}
\begin{proof} 
The equivalence $(ii) \Leftrightarrow (iii)$ is well known, e.g., \cite[Theorem 1]{LO93}.
The equivalence $(i) \Leftrightarrow (ii)$ is not hard. 
The key observation is $\epsilon_i - \epsilon_j \ge \epsilon_p - \epsilon_q$ if and only if $i \le p < q \le j$.

 \end{proof}

\section{Application to graph polynomials of Cocomparability graph}
\label{sec:Comparability}
 
Owing to Theorem \ref{thm:cocomp-characterize}, we will be able to give new formulas for the chromatic polynomial and the (reduced) graphic Eulerian polynomial of cocomparability graphs.

Let $G = ([\ell],\E)$ be a simple graph, i.e., no loops and no multiple edges. 
Let $c_G(t)$ be the chromatic polynomial of $G$. 
The \emph{graphic Eulerian polynomial} of $G$ is the polynomial $W_G(t)$ defined by 
\begin{equation*} \label{eq:E-pol}
\sum_{q \ge 0}c_G(q)t^q = \frac{W_G(t)}{(1-t)^{ \ell+1}}.
\end{equation*}
The coefficients of  $W_G(t)$ are proved to be nonnegative integers, and have various combinatorial interpretations, e.g., \cite{B92, CG95, S01}. 
Let us recall one of its interpretations following the last two references. 
Denote by $\mathfrak{S}_\ell$ the symmetric group on $[\ell]$.
\begin{definition} \label{def:g-descents-rank}
For $\pi=\pi_1\ldots \pi_\ell  \in \mathfrak{S}_\ell $, the \emph{rank} $\rho(\pi_i)$ of $\pi_i$ ($1 \le i \le \ell$) is defined to be the largest integer $r$ so that there are values $1 \le i_1 < i_2 <\cdots<i_r=i$ with 
$\{\pi_{i_j}, \pi_{i_{j+1}}\}\in \E$ for all $1 \le j \le r-1$.
We say that $\pi \in \mathfrak{S}_\ell $ has a \emph{graphic descent at $i\in[ \ell -1]$}  if either (1) $\rho(\pi_i) > \rho(\pi_{i+1})$, or (2) $\rho(\pi_i) =\rho(\pi_{i+1})$ and $\pi_i > \pi_{i+1}$.
The \emph{graphic Eulerian numbers} $w_k(G)$ $(1 \le k \le  \ell )$ are defined by
$$
    w_k(G)=   \#\{ \pi \in \mathfrak{S}_\ell \mid \mbox{$\pi$ has exactly $ \ell -k$ graphic descents}\}.
$$
 \end{definition}
 
\begin{theorem}\label{thm:gEp}
We have
$W_G(t)= \sum_{k=1}^{ \ell } w_k(G)t^k.$
 Equivalently, 
$$
c_G(t) = \sum_{k=1}^{\ell} w_{k}(G) \binom{t+ \ell -k}{\ell}.
$$
\end{theorem}
\begin{proof} 
See, e.g., \cite[Theorem 2]{CG95}, \cite[Theorem 6]{S01}. 
The equivalence of the formulas holds true in a more general setting, e.g., \cite[Theorem 2.1]{B92}.
\end{proof}

 \begin{remark}
\label{rem:F-edgeless-complete}
 If $G$ is the empty graph, then the graphic descent is the ordinary descent, i.e., the index $i\in[\ell -1]$  such that $\pi_i > \pi_{i+1}$. 
In this case, $W_G(t)$ is known as the \emph{classical $\ell$-th Eulerian polynomial}, which  first appeared in a work of Euler \cite{E36}. 
\end{remark}

It is a standard fact that $c_G(t)$ is divisible by $t$. The \emph{reduced graphic Eulerian polynomial} of $G$ is the polynomial $Y_G(t)$ defined by 
\begin{equation*} \label{eq:r-E-pol}
\sum_{q \ge 1}\frac{c_G(q)}{q}t^q = \frac{Y_G(t)}{(1-t)^{\ell}}.
\end{equation*}
The polynomial $Y_G(t)$ also has nonnegative integer coefficients, e.g., \cite[Theorem 3.5]{J05}.  
More precisely, $Y_G(t)$ is interpreted in terms of the $h$-polynomial of a certain \emph{relative complex}. 
However, unlike $W_G(t)$, less seems to be known about combinatorial interpretations of the coefficients of $Y_G(t)$. 
We will give a combinatorial interpretation for $Y_G(t)$ when $G$ is a cocomparability graph (Theorem \ref{thm:shift-typeA}).
One can readily compute $Y_G(t)$ from $W_G(t)$ and vice versa as we will see below.
Write $Y_G(t) = \sum_{k=1}^{\ell} y_k(G)t^k$ (note that $Y_G(0) =0$).

\begin{proposition}\label{prop:relation}   
The polynomials $W_G(t)$ and $Y_G(t)$ satisfy the \emph{Eulerian recurrence}, i.e.,
\begin{equation*} \label{eq:suffice}
W_G(t) = t(1-t)\frac{d}{dt}Y_G(t)+t \ell Y_G(t).
\end{equation*}
Equivalently,  for every $1 \le k \le \ell $,
\begin{equation*} \label{eq:unnatural}
w_k(G) = ky_k(G)  +( \ell -k+1)y_{k-1}(G).
\end{equation*}
Conversely, if $F(t)$ is a polynomial such that $F(0)=0$, and $W_G(t)$ and $F(t)$ satisfy the Eulerian recurrence, then $F(t) = Y_G(t)$.
\end{proposition} 
\begin{proof} 
Straightforward. 
\end{proof}

Now to link the Eulerian polynomials above with  the cocomparability graphs, we need to recall the notion of $\A$-Eulerian polynomial following \cite[\S4]{ATY20}. 
We will focus only on type $A$.
Let $G^c = (\V, \E(G^c))$ be the complement graph of $G$. 
 \begin{definition}\label{def:ek}
 We say that $\pi=\pi_1\ldots \pi_\ell  \in \mathfrak{S}_\ell $ has a \emph{$\A$-descent (w.r.t. $G$)} at $i\in[\ell]$ if $\pi_i> \pi_{i+1}$ and $\{\pi_i, \pi_{i+1}\}\in \E(G^c)$ ($\pi_{ \ell+1}= \pi_1$).
For $0 \le k \le \ell $, define
$$
    f_k(G)=   \frac1{\ell}\#\{ \pi \in \mathfrak{S}_\ell \mid \mbox{$\pi$ has exactly $ \ell -k$ $\A$-descents}\}.
$$
It is easily seen that $f_0(G)=0$. 
Set $F_G(t) := \sum_{k=1}^{\ell} f_k(G)t^k$. 
\end{definition}
Recall the notation $\Psi(G) = \{  \epsilon_i-\epsilon_j  \mid \{i,j\}  \in \E \,( i <  j)\} \subseteq \Phi^+(A_{\ell-1})$ in \S\ref{sec:intro}.
It is not hard to check that $F_{G}(t)$ equals the $\A$-Eulerian polynomial  \cite[Definition 4.2]{ATY20} of $\Psi(G)$.
If $G$ is the empty graph, then $F_G(t)$ is the classical $( \ell -1)$-th Eulerian polynomial.
Unlike $W_G(t)$ or $Y_G(t)$,  $F_G(t)$ is not a graph invariant in the sense that it depends on the labeling of $G$  (see Examples \ref{ex:G} and \ref{ex:G'}). 
The proposition below says computing $F_G(t)$ requires only $( \ell -1)!$ permutations.
\begin{proposition}
\label{prop:positive}   
If $1 \le k \le \ell $, then
$$
    f_k(G)=\#\{ \pi \in \mathfrak{S}_\ell \mid \mbox{$\pi_1=\ell$ and $\pi$ has exactly $ \ell -k$ $\A$-descents}\}.
$$
 \end{proposition} 
\begin{proof} 
When $G$ is the empty graph on $[\ell]$, it is already proved in \cite[Exercise 1.11]{P15}. 
The same argument applies to the general case. 
Let $\tau=23\ldots \ell1 \in \mathfrak{S}_\ell $, then $\tau$ cyclically shifts the numbers $1$ to $ \ell $. 
The action of $\tau$ on $\mathfrak{S}_\ell $ by right multiplication partitions $\mathfrak{S}_\ell $ into $( \ell -1)!$ orbits of size $\ell$. 
The number of $\A$-descents is constant on each orbit. 
Also, we can choose the permutation $\pi$ with $\pi_1=\ell$ to be a representative of each orbit.
 \end{proof}
 
The following is the main result of this section. 
 \begin{theorem}\label{thm:shift-typeA}   
Let $G=([\ell],\E)$ be a  graph. 
The following  are equivalent.
\begin{enumerate}[(i)]
\item $G$ is a cocomparability graph and the ordering $1<2<\cdots<\ell$ is an umbrella-free ordering.
\item The ordering $1<2<\cdots<\ell$ is an umbrella-free ordering.
    \item $\Psi(G) \subseteq \Phi^+(A_{\ell-1})$ is compatible (= strongly compatible). 
    \item   The chromatic polynomial of $G$ is given by 
    $$c_G(t) =t\sum_{k=1}^{n} f_k(G) {t+\ell-1-k \choose \ell-1}.$$
    \item $F_{G}(t)$ equals the reduced graphic Eulerian polynomial $G$, i.e.,
        $$F_{G}(t) = Y_{G}(t)=(1-t)^\ell\sum_{q \ge 1}\frac{c_G(q)}{q}t^q,$$
Equivalently, $F_{G}(t) = Y_{K}(t),$ where $K$ is any graph isomorphic to $G$.
   \item The polynomials $W_{G}(t)$ and $F_{G}(t)$ satisfy the Eulerian recurrence, i.e.,
\begin{equation*} \label{eq:suffice-thm}
W_{G}(t) = t(1-t)\frac{d}{dt}F_{G}(t)+t \ell F_{G}(t)
\end{equation*}
Equivalently,  for every $1 \le k \le \ell $,
\begin{equation*} \label{eq:unnatural-thm}
w_k(G) = kf_k(G)  +( \ell -k+1)f_{k-1}(G).
\end{equation*}
    \end{enumerate}
\end{theorem} 
\begin{proof} 
The equivalences $(i) \Leftrightarrow (ii)\Leftrightarrow (iii)$ follow from Corollary \ref{cor:c-A} and Lemma \ref{lem:cocomp-characterize1}.
The statements (iii), (iv) and (v) are simply those in  \cite[Theorems 4.11 and 4.24]{ATY20} when restricting the root system to type $A$, hence they are equivalent. 
The equivalence  $(v) \Leftrightarrow (vi)$ follows from Proposition \ref{prop:relation} and the fact that $F_G(0)=0$ (Definition \ref{def:ek}). 
 \end{proof}
 
Thus, to compute $c_G(t)$, $Y_{G}(t)$ and $W_{G}(t)$ for a given cocomparability graph $G$, in principle, we can do as follows: find an umbrella-free ordering $v_1 < v_2  < \cdots  <  v_\ell$ of its vertices (which can be done in linear time \cite{McS99}), label each vertex $v_i$ by $i\in[\ell]$, and apply Theorem \ref{thm:shift-typeA}.

\begin{example}
\label{ex:G}
Let $G$ be a graph given in Figure \ref{fig:4}. 
Thus, $\Psi(G)= \{\alpha_1, \alpha_2\} \subseteq \Phi^+(A_3)$, which is an ideal, hence compatible (Theorem \ref{thm:ideal-strong-com}).
The computation of $F_G(t)$ and $W_G(t)$ is illustrated in Table \ref{tab:compute}.
Let $\tau=2341$. 
There are $8$ permutations having $2$ $\A$-descents and $16$ having  $1$ $\A$-descent. Thus, $F_G(t)=4t^3+2t^2$. 
There are $4$ permutations having $2$ graphic descents (cells in green), $4$ having $0$ graphic descent (cells in red), and the remaining $16$ have $1$  graphic descent. 
Thus, $W_G(t)=4t^4+16t^3+2t^2$. 
The calculation is consistent with Theorem \ref{thm:shift-typeA}.
\end{example}

\begin{table}[htbp]
\centering
{\footnotesize\renewcommand\arraystretch{1.5} 
\begin{tabular}{|c|c|c|c|c|}
\hline
\mbox{$\#\A$-descents} & $\pi$ & $\pi\tau$   & $\pi\tau^2$ & $\pi\tau^3$  \\
\hline\hline
  \multirow{2}{*}{$2$} &  \cellcolor{green!50}4312 & \cellcolor{green!50}3124  & 1243 & 2431  \\
  & \cellcolor{green!50}4231 & \cellcolor{green!50}2314  & 3142 &  \cellcolor{red!25}1423 \\
\hline
\multirow{4}{*}{$1$} & 4213 & 2134  & \cellcolor{red!25}1342  &   \cellcolor{red!25} 3421  \\
& 4132 & 1324  &  3241 &  \cellcolor{red!25}2413   \\
&  4321 & 3214  & 2143  &  1432  \\
& 4123 & 1234 &  2341 & 3412   \\
\hline
\end{tabular}
}
\bigskip
\caption{Computation of $F_G(t)$ and $W_G(t)$ for the graph $G$ in Figure \ref{fig:4}.}
\label{tab:compute}
\end{table}

\begin{example}
\label{ex:G'}
Let $G'$ be a graph given in Figure \ref{fig:4}. 
Note that $G$ and $G'$ have isomorphic underlying unlabeled graphs.
But $\Psi(G')$ is not compatible, which can be checked  either by Theorem \ref{thm:shift-typeA}(v) ($Y_{G'}(t) \ne F_{G'}(t)$), or by  Theorem \ref{thm:shift-typeA}(ii) ($1<2<4$, $\{1,4\} \in \E$ but $\{1,2\} \notin \E$, $\{2,4\} \notin \E$).
\end{example}

\begin{example}
\label{ex:Claw}
Let $G$ be the claw $K_{1,3}$ given in Figure \ref{fig:4}. 
By Theorem \ref{thm:shift-typeA}(ii), $\Psi(K_{1,3})$ is compatible. 
By Theorem  \ref{thm:ideal-characterize}, $\Psi(K_{1,3})$  is not an ideal w.r.t. any positive system since $K_{1,3}$ is not an interval graph. 
\end{example}

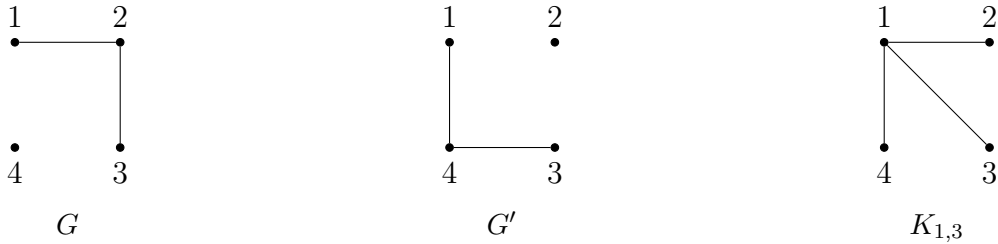
\begin{figure}[htbp]
\centering
\begin{subfigure}{.35\textwidth}
  \centering
\begin{tikzpicture}[scale=.7,colorstyle/.style={circle, draw=black!100,fill=black!100, inner sep=0pt, minimum size=1mm}]
    \node (1) at (-1,1)[colorstyle,label=above:$1$]{};
    \node (2) at (1,1)[colorstyle,label=above:$2$]{};
    \node (3) at (1,-1)[colorstyle,label=below:$3$]{};
    \node (4) at (-1,-1)[colorstyle,label=below:$4$]{};
 \draw (2) edge [-] (1);
    \draw[](3) edge [-](2);
\end{tikzpicture}
  \caption*{$G$}
  \label{fig:sub1}
\end{subfigure}%
\begin{subfigure}{.35\textwidth}
  \centering
\begin{tikzpicture}[scale=.7,colorstyle/.style={circle, draw=black!100,fill=black!100, inner sep=0pt, minimum size=1mm}]
    \node (1) at (-1,1)[colorstyle,label=above:$1$]{};
    \node (2) at (1,1)[colorstyle,label=above:$2$]{};
    \node (3) at (1,-1)[colorstyle,label=below:$3$]{};
    \node (4) at (-1,-1)[colorstyle,label=below:$4$]{};
 \draw (4) edge [-] (1);
    \draw(3) edge [-](4);
\end{tikzpicture}
  \caption*{$G'$}
  \label{fig:sub2}
\end{subfigure}%
\begin{subfigure}{.35\textwidth}
  \centering
\begin{tikzpicture}[scale=.7,colorstyle/.style={circle, draw=black!100,fill=black!100, inner sep=0pt, minimum size=1mm}]
    \node (1) at (-1,1)[colorstyle,label=above:$1$]{};
    \node (2) at (1,1)[colorstyle,label=above:$2$]{};
    \node (3) at (1,-1)[colorstyle,label=below:$3$]{};
    \node (4) at (-1,-1)[colorstyle,label=below:$4$]{};
    \draw[](1) edge [-](4);
    \draw[](2) edge [-](1);
    \draw[](3) edge [-](1);
\end{tikzpicture}
  \caption*{$K_{1,3}$}
  \label{fig:sub3}
\end{subfigure}
\caption{Three cocomparability graphs with $4$ vertices.}
\label{fig:4}
\end{figure}

 
\vn
\textbf{Acknowledgements.}
We would like  to thank Ahmed Umer Ashraf and Masahiko Yoshinaga for stimulating conversations on the graphic Eulerian polynomials.
We would also like to thank Shuhei Tsujie for valuable discussion and for letting us know the characterization of  ideal-graphic arrangements in terms of unit interval graphs (Theorem \ref{thm:ideal-characterize}). 
The first author is supported by JSPS Research Fellowship for Young Scientists Grant Number 19J12024.
The second author was partially supported by JSPS KAKENHI 19J00312 and 19K14505.
\bibliographystyle{alpha} 
\bibliography{references}

\end{document}